\documentclass[11pt]{article}
\usepackage{amsthm,amssymb,amsmath,amscd}
\usepackage[all]{xy}
\usepackage{mathrsfs}
\usepackage{multirow}
\usepackage{mathptm}

\newtheorem{Definition}{Definition}[section]
\newtheorem{Proposition}[Definition]{Proposition}
\newtheorem{Theorem}[Definition]{Theorem}
\newtheorem{Lemma}[Definition]{Lemma}
\newtheorem{Corollary}[Definition]{Corollary}

\newcommand{\add}{{\rm add}}
\newcommand{\con}{{\rm con}}
\newcommand{\domdim}{{\rm dom.dim}}
\newcommand{\Hom}{{\rm Hom }}

\newcommand{\soc}{{\rm soc}}
\renewcommand{\top}{{\rm top}}

\newcommand{\nustHom}{{\rm \underline{Hom}}^{\nu}}
\newcommand{\repdim}{{\rm rep.dim }}
\newcommand{\gldim}{{\rm gl.dim}}
\newcommand{\findim}{{\rm fin.dim}}
\newcommand{\End}{{\rm End}}
\newcommand{\Ext}{{\rm Ext}}
\newcommand{\Coker}{{\rm Coker}}
\newcommand{\Ker}{{\rm Ker}}
\newcommand{\im}{{\rm Im}\,}
\newcommand{\cpx}[1]{#1^{\bullet}}
\newcommand{\D}[1]{\mathscr{D}(#1)}

\newcommand{\Db}[1]{ \mathscr{D}^{\rm b}(#1)}
\newcommand{\C}[1]{\mathscr{C}(#1)}

\newcommand{\Cb}[1]{{\mathscr{C}^b}(#1)}
\newcommand{\K}[1]{\mathscr{K}(#1)}

\newcommand{\Kb}[1]{ \mathscr{K}^{\rm b}(#1)}
\newcommand{\modcat}[1]{#1\mbox{{\rm -mod}}}
\newcommand{\stmodcat}[1]{#1\mbox{{\rm -{\underline{mod}}}}}
\newcommand{\stimodcat}[1]{#1\mbox{-}{\overline{\rm mod}}}
\newcommand{\pmodcat}[1]{#1\mbox{{\rm -proj}}}
\newcommand{\imodcat}[1]{#1\mbox{{\rm -inj}}}
\newcommand{\nust}[1]{\nu_{#1}\mbox{{\rm -Stp}}}
\newcommand{\nustcat}[1]{#1\mbox{{\rm -\underline{mod}}}^{\nu}}
\newcommand{\nustmor}[1]{\underline{#1}}

\newcommand{\Ex}[3]{{\rm E}_{#1}^{#2}(#3)}

\renewcommand{\leq}{\leqslant}
\renewcommand{\geq}{\geqslant}

\newcommand{\lra}{\longrightarrow}

\newcommand{\lraf}[1]{\stackrel{#1}{\lra}}

\newcommand{\llaf}[1]{\stackrel{#1}{{\Longleftarrow}}}
\newcommand{\rlaf}[1]{\stackrel{#1}{\Longrightarrow}}

\title{On iterated almost $\nu$-stable derived equivalences}
\author{Wei Hu}
\date{}

\begin{document}
\maketitle

\begin{abstract}
    In a recent paper \cite{HuXi3}, we introduced a classes of
    derived equivalences called almost $\nu$-stable derived
    equivalences. The crucial property is that an almost
    $\nu$-stable derived equivalence always induces a stable
    equivalence of Morita type, which generalizes a well-known
    result of Rickard: derived-equivalent self-injective algebras
    are stably equivalent of Morita type. In this paper, we shall
    consider the compositions of almost $\nu$-stable derived
    equivalences and their quasi-inverses, which are called iterated
    almost $\nu$-stable derived equivalences. We
    give a sufficient and necessary condition for
    a derived equivalence to be an iterated almost $\nu$-stable
    derived equivalence, and give an explicit construction of the stable equivalence functor
    induced by an iterated almost $\nu$-stable derived equivalence.
    As a consequence, we get some new sufficient
    conditions for a derived equivalence between general
    finite-dimensional algebras to induce a stable equivalence of
    Morita type.
\end{abstract}

\renewcommand{\thefootnote}{\alph{footnote}}
\renewcommand{\thefootnote}{\alph{footnote}}
\setcounter{footnote}{-1} 
\renewcommand{\thefootnote}{\alph{footnote}}
\setcounter{footnote}{-1} 

\section{Introduction}

In \cite{HuXi3}, we introduced a class of derived equivalences
called almost $\nu$-stable derived equivalences. The crucial
property is that an almost $\nu$-stable derived equivalence always
induces a stable equivalence of Morita type, which generalizes a
classical result of Rickard (\cite[Corollary 5.5]{RickDFun}). This
also gives a sufficient condition for a derived equivalence between
general finite-dimensional algebras to induce a stable equivalence
of Morita type. Note that many homological dimensions, such as
global dimension, finitistic dimension, and representation
dimension, are not invariant under derived equivalences in general.
But they are all preserved by stable equivalences of Morita type.
So, this also helps us to compare the homological dimensions of
derived-equivalent algebras.

 Let us first recall the definition of almost
 $\nu$-stable derived equivalences. Let $F:\Db{A}\lra\Db{B}$ be a
derived equivalence between two Artin algebras $A$ and $B$, where
$\Db{A}$ and $\Db{B}$ stand for the derived categories of bounded
complexes over $A$ and $B$, respectively. We use $F^{-1}$ to denote
a quasi-inverse of $F$.  $F$ is called an {\em almost $\nu$-stable
derived equivalence} if the following hold:

 \smallskip
   (1) {\it The tilting complex $\cpx{T}$ associated to $F$ has the
   following form: }
   $$0\lra T^{-n}\lra\cdots\lra T^{-1}\lra T^0\lra 0$$
In this case, the tilting $\cpx{\bar{T}}$ associated to $F^{-1}$ has
the following form (see \cite[Lemma 2.1]{HuXi3}):
$$0\lra \bar{T}^{0}\lra\bar{T}^1\lra\cdots\lra \bar{T}^{n}\lra 0$$

   (2) {\it  $\add(\bigoplus_{i=1}^nT^{-i})=\add(\bigoplus_{i=1}^n\nu_AT^{-i})$
   and
   $\add(\bigoplus_{i=1}^n\bar{T}^{i})=\add(\bigoplus_{i=1}^n\nu_B\bar{T}^{i})$, where $\nu$ is the Nakayama functor.}

\medskip
Let us remark that the composition of two almost $\nu$-stable
derived equivalences (or their quasi-inverses) is no longer almost
$\nu$-stable in general. If a derived equivalence is a composition
$F\simeq F_1F_2\cdots F_m$ with $F_i$ or $F_i^{-1}$ being an almost
$\nu$-stable derived equivalence for all $i$, then $F$ is called an
{\em iterated almost $\nu$-stable derived equivalence}. By
definition, we see that an almost $\nu$-stable derived equivalence
and its quasi-inverse are iterated almost $\nu$-stable derived
equivalences, and that the composition of two iterated almost
$\nu$-stable derived equivalences is again an iterated almost
$\nu$-stable derived equivalence. Clearly, an iterated almost
$\nu$-stable derived equivalence always induces a stable equivalence
of Morita type, and therefore the involved algebras have many common
homological dimensions. But the problem is:

\medskip
{\parindent=0pt\bf Question: } {\it Given a derived equivalence $F$,
how to determine whether $F$ is iterated almost $\nu$-stable or
not?}

\medskip
The main purpose of this note is to give a complete answer to the
above question. For a bounded complex $\cpx{X}$ over an algebra $A$,
we use $X^{\pm}$ to denote $\bigoplus_{i\neq 0}X^i$. The main result
of this note can be stated as the following theorem.
\begin{Theorem}
  Let $F:\Db{A}\lra\Db{B}$ be a derived equivalence between two
  Artin
  algebras $A$ and $B$. Suppose that
  $\cpx{T}$ and $\cpx{\bar{T}}$ are the tilting complexes associated
  to $F$ and $F^{-1}$, respectively.
   Then $F$ is an iterated almost
  $\nu$-stable derived equivalence if and only if $\add(_AT^{\pm})=\add(\nu_AT^{\pm})$ and
  $\add(_B\bar{T}^{\pm})=\add(\nu_B\bar{T}^{\pm})$.\label{Theorem 1}
\end{Theorem}

The above theorem tells us that, by checking the terms of tilting
complexes, we can determine whether a derived equivalence is
iterated almost $\nu$-stable or not. If we work with
finite-dimensional algebras over a field, then we have several other
characterizations of iterated almost $\nu$-stable derived
equivalences. For details, see Theorem \ref{TheoremOtherChar} below.
As a consequence of Theorem \ref{Theorem 1}, we have the following
corollary, which provides a new sufficient condition for a derived
equivalence to induce a stable equivalence of Morita type. For
information on stable equivalences of Morita type, we refer to
\cite{Broue1994,LiuXi1,LiuXi2,LiuXi3}.
\begin{Corollary}
    Let $F:\Db{A}\lra\Db{B}$ be a derived equivalence between two
    finite-dimensional algebras $A$ and $B$ over a field. Suppose that
  $\cpx{T}$ and $\cpx{\bar{T}}$ are the tilting complexes associated
  to $F$ and $F^{-1}$, respectively. If $\add(_AT^{\pm})=\add(\nu_AT^{\pm})$ and $\add(_B\bar{T}^{\pm})=\add(\nu_B\bar{T}^{\pm})$, then $A$ and $B$ are
   stably equivalent of Morita type. \label{Corollary 2}
\end{Corollary}

This paper is organized as follows. In Section 2, we shall recall
some notations and basic facts. Theorem \ref{Theorem 1} will be
proved in Section 3 after several lemmas. Section 4 is devoted to
describing the stable equivalence functor induced by an iterated
almost $\nu$-stable derived equivalence. Finally, in Section 5, we
shall give several methods to construct iterated almost $\nu$-stable
derived equivalences.

\section{Preliminaries}
In this section, we shall recall some basic definitions and facts
needed in our later proofs.

Throughout this paper, unless specified otherwise, all algebras will
be Artin algebras over a fixed commutative Artin ring $R$. All
modules will be finitely generated unitary left modules. For an
algebra $A$, the category of $A$-modules is denoted by $\modcat{A}$;
the full subcategory of $A$-mod consisting of projective modules is
denoted by $\pmodcat{A}$. The stable module category, denoted by
$\stmodcat{A}$, is the quotient category of $\modcat{A}$ modulo the
ideal generated by morphisms factorizing through projective modules.
We denote by $\nu_A$ the usual Nakayama functor.

Let $\cal C$ be an additive category. The composition of two
morphisms $f: X\lra Y$ and $g: Y\lra Z$ in $\cal C$ will be denoted
by $fg$. For two functors $F:\mathcal{C}\rightarrow \mathcal{D}$ and
$G:\mathcal{D}\rightarrow\mathcal{E}$ of categories, their
composition is denoted by $GF$. For an object $X$ in $\mathcal{C}$,
$\add(X)$ is the full subcategory of $\cal C$ consisting of all
direct summands of finite direct sums of copies of $X$.

A complex $\cpx{X}$ over $\cal C$ is a sequence  $\cdots\lra
X^{i-1}\lraf{d_X^{i-1}} X^i\lraf{d_X^i}
X^{i+1}\lraf{d_X^{i+1}}\cdots$ in $\cal C$ such that
$d_X^id_X^{i+1}=0$ for all integers $i$. The category of complexes
over ${\cal C}$ is denoted by $\C{\cal C}$. The homotopy category of
complexes over $\mathcal{C}$ is denoted by $\K{\mathcal C}$. When
$\cal C$ is an abelian category, the derived category of complexes
over $\cal C$ is denoted by $\D{\cal C}$. The full subcategory of
$\K{\cal C}$ and $\D{\cal C}$ consisting of bounded complexes over
$\mathcal{C}$ is denoted by $\Kb{\mathcal C}$ and $\Db{\mathcal C}$,
respectively. As usual, for a given algebra $A$, we simply write
$\Kb{A}$ and $\Db{A}$ for $\Kb{\modcat{A}}$ and $\Db{\modcat{A}}$,
respectively.

It is well-known that, for an algebra $A$, $\Kb{A}$ and $\Db{A}$ are
triangulated categories. For basic results on triangulated
categories, we refer to Happel's book \cite{HappelTriangle}.
Throughout this paper, we use $\cpx{X}[n]$ to denote the complex
obtained by shifting $\cpx{X}$ to the left by $n$ degree.

Let $A$ be an algebra. A homomorphism $f: X\lra Y$ of $A$-modules is
called a \emph{radical map} if, for any module $Z$ and homomorphisms
$h: Z\lra X$ and $g: Y\lra Z$, the composition $hfg$ is not an
isomorphism. A complex over $\modcat{A}$ is called a \emph{radical}
complex if all its differential maps are radical maps. Every complex
over $\modcat{A}$ is isomorphic in the homotopy category $\K{A}$ to
a radical complex. It is easy to see that if two radical complexes
$\cpx{X}$ and $\cpx{Y}$  are isomorphic in $\K{A}$, then $\cpx{X}$
and $\cpx{Y}$ are isomorphic in $\C{A}$.

Two algebras $A$ and $B$ are said to be \emph{derived-equivalent} if
their derived categories $\Db{A}$ and $\Db{B}$ are equivalent as
triangulated categories. In \cite{RickMoritaTh}, Rickard proved that
two algebras are derived-equivalent if and only if there is a
complex $\cpx{T}$ in $\Kb{\pmodcat{A}}$ satisfying

\smallskip
(1) $\Hom(\cpx{T},\cpx{T}[n])=0$ for all  $n\ne 0$, and

(2) $\add(\cpx{T})$ generates $\Kb{\pmodcat{A}}$ as a triangulated
category

\smallskip
{\parindent=0pt such that $B\simeq\End(\cpx{T})$.} A complex in
$\Kb{\pmodcat{A}}$ satisfying the above two conditions is called a
{\em tilting complex} over $A$.
It is known that, given a derived equivalence $F$ between $A$ and
$B$, there is a unique (up to isomorphism) tilting complex $\cpx{T}$
over $A$ such that $F(\cpx{T})\simeq B$. If $\cpx{T}$ is a radical
complex, it is called a tilting complex \emph{associated} to $F$.
Note that, by definition, a tilting complex associated to $F$ is
unique up to isomorphism in $\Cb{A}$.

The following lemma is useful in our later proof. For the
convenience of the reader, we provide  a proof.
\begin{Lemma}
  Let ${\cal C}$ and ${\cal D}$ be two additive categories, and let $F:\Kb{\cal C}\lra\Kb{\cal D}$
  be a triangle functor. Let
   $\cpx{X}$ be a complex in $\Kb{\cal C}$. For each term $X^i$, let
   $\cpx{Y_i}$ be a complex isomorphic to $F(X^i)$. Then $F(\cpx{X})$ is
   isomorphic to a complex $\cpx{Z}$ with
   $Z^m=\bigoplus_{i+j=m}Y_i^j$ for all $m\in\mathbb{Z}$. \label{lemmaFormofFX}
\end{Lemma}
\begin{proof}
  We use induction on the number of non-zero terms of $\cpx{X}$. If
  $\cpx{X}$ has only one non-zero term, then it is obvious. Assume
  that $\cpx{X}$ has more than one non-zero terms. Without loss of
  generality, we suppose that $\cpx{X}$ is the following complex
  $$0\lra X^0\lra X^1\lra\cdots\lra X^n\lra 0$$
  with $X^i\neq 0$ for all $i=0, 1, \cdots, n$. Let $\sigma_{\geq
  1}\cpx{X}$ be the complex $0\lra X^1\lra\cdots\lra X^n\lra 0$.
  Then there is a distinguished triangle in $\Kb{\cal C}$:
   $$X^0[-1]\lra \sigma_{\geq 1}\cpx{X}\lra \cpx{X}\lra X^0.$$
Applying $F$, we get a distinguished triangle in $\Kb{\cal D}$:
  $$F(X^0[-1])\lra F(\sigma_{\geq 1}\cpx{X})\lra F(\cpx{X})\lra
  F(X^0).$$
By induction, $F(\sigma_{\geq 1}\cpx{X})$ is isomorphic to a complex
$\cpx{U}$ with $U^m=\bigoplus\limits_{1\leq i\leq n, i+j=m}Y_i^j$.
Thus, $F(\cpx{X})$ is isomorphic to the mapping cone $\cpx{Z}$ of
the map from $\cpx{Y_0}[-1]$ to $\cpx{U}$. Thus, by definition, we
have
$$Z^{m}=\bigoplus_{0\leq i\leq n, i+j=m} Y_i^j=\bigoplus_{i+j=m}Y_i^j.$$
 This finishes the proof.
\end{proof}

{\it Remark: } Let $F:\Db{A}\lra\Db{B}$ be a derived equivalence
between two algebras $A$ and $B$. $F$ induces an equivalence
$F:\Kb{\pmodcat{A}}\lra\Kb{\pmodcat{B}}$. So, for a bounded complex
of projective $A$-modules, we can use the above lemma to calculate
its image under $F$.

\section{Characterizations of iterated almost $\nu$-stable derived equivalences}

In this section, we shall give a proof of our main result Theorem
\ref{Theorem 1}, which characterizes iterated almost $\nu$-stable
derived equivalences in terms of tilting complexes. In case that the
algebras are finite-dimensional algebras over a field, we shall give
several other characterizations of iterated almost $\nu$-stable
derived equivalences. For this purpose, we need some lemmas.

Let $A$ be an algebra, and let $_AE$ be the direct sum of all those
non-isomorphic indecomposable projective $A$-modules $P$ with
$\nu^i_AP$ being projective-injective for all $i\geq 0$. The
$A$-module $_AE$ is unique up to isomorphism, and is called the {\em
maximal $\nu$-stable } $A$-module. If $_AQ$ is a projective
$A$-module such that $\add(_AQ)=\add(\nu_AQ)$, then clearly
$_AQ\in\add(_AE)$. Throughout this paper, we use $\nust{A}$ to
denote the category $\add(_AE)$.  Recall that for a bounded complex
$\cpx{X}$ over $A$, we use $X^{\pm}$ to denote the $A$-module
$\bigoplus_{i\neq 0}X^i$.
\begin{Lemma}
  Let $\cpx{T}$ be a tilting complex associated to a derived
  equivalence $F:\Db{A}\lra\Db{B}$ between two algebras. Then the following two conditions are equivalent.

  $(1)$ $\add(\nu_AT^{\pm})=\add(_AT^{\pm})$;

  $(2)$ $_AT^{\pm}\in\nust{A}$.\label{equivCond}
\end{Lemma}
\begin{proof}
  Clearly, we have $(1)\Rightarrow (2)$. It remains
  to show that $(2)$ implies $(1)$. Now we assume $(2)$ holds. Let
  $Q_1=\bigoplus_{i<0}T^i$.  Using the same method in the proof of \cite[Lemma 3.1]{HuXi3},
  $F^{-1}(B)$ is isomorphic in $\Db{A}$ to a complex $\cpx{X}$
  with $X^i\in\add(\nu_AQ_1)$ for all $i<0$. Thus,
  $\cpx{T}\simeq\cpx{X}$, and there is a quasi-isomorphism $\cpx{f}:
  \cpx{T}\lra\cpx{X}$, which induces a quasi-isomorphism
  $$\xymatrix@C=10mm@R=7mm{
    \cpx{U}: \quad \cdots \ar[r] & T^{-2}\ar[d]^{f^{-2}}\ar[r]^{d_T^{-2}} &
    T^{-1}\ar[d]^{f^{-1}}\ar[r]^{\pi_T}& \im d_T^{-1}\ar[d]^{f^0|_{\im d_T^{-1}}}\ar[r] & 0\\
    \cpx{V}: \quad \cdots \ar[r] & X^{-2}\ar[r]^{d_X^{-2}} &
    X^{-1}\ar[r]^{\pi_X}& \im d_X^{-1}\ar[r] & 0.\\
   }$$
We claim that the canonical epimorphism $\pi_T: T^{-1}\lra \im
d_T^{-1}$ is still a radical map. Otherwise, let $h: Y\lra T^{-1}$
and $g: \im d_T^{-1}\lra Y$ be such that $h\pi_Tg=1_Y$. Then $Y$ is
isomorphic to a direct summand of $T^{-1}$, and therefore $Y$ is an
injective module. Thus, $g$ factors through the inclusion $\lambda:
\im d_T^{-1}\lra T^0$, say $g=\lambda u$. Consequently
$1_Y=h\pi_T\lambda u=hd_T^{-1}u$. This means that $d_T^{-1}:
T^{-1}\lra T^0$ is not radical which is a contradiction. Since $T^i$
and $X^i$ are injective for all $i<0$, by \cite[Lemma 2.2]{HuXi3},
$\cpx{U}$ and $\cpx{V}$ are isomorphic in $\Kb{A}$. Thus, $T^{i}$ is
a direct summand of $X^{i}$ for all $i<0$, and consequently
$Q_1=\bigoplus_{i<0}T^i\in\add(\nu_AQ_1)$. Since $Q_1$ and
  $\nu_AQ_1$ have the same number of non-isomorphic indecomposable
  direct summands, we have $\add(_AQ_1)=\add(\nu_AQ_1)$. Let
$Q_2:=\bigoplus_{i>
  0}T^i$. Similarly, we have $\add(_AQ_2)=\add(\nu_AQ_2)$. Consequently, $\add(_AT^{\pm})=\add(_AQ_1\oplus {}_AQ_2)=\add(\nu_AQ_1\oplus\nu_AQ_2)=\add(\nu_AT^{\pm})$.  Hence $(2)\Rightarrow
  (1)$.
\end{proof}

In the following, we shall  use Lemma \ref{equivCond} freely. For
instance, in the definition of an almost $\nu$-stable equivalence,
the condition
$\add(\bigoplus_{i=1}^nT^{-i})=\add(\bigoplus_{i=1}^n\nu_AT^{-i})$
is equivalent to say that $T^{-i}\in\nust{A}$ for all
$i=1,\cdots,n$.

\begin{Lemma}
   Let $F: \Db{A}\lra\Db{B}$ be a derived equivalence between two
 algebras $A$ and $B$, and let $\cpx{T}$ and $\cpx{\bar{T}}$ be
   the tilting complexes associated to $F$ and $F^{-1}$,
   respectively. If $\add(_AT^{\pm})=\add(\nu_AT^{\pm})$
   and
   $\add(_B\bar{T}^{\pm})=\add(\nu_B\bar{T}^{\pm})$,
   then $F$ induces an equivalence between $\Kb{\nust{A}}$ and
   $\Kb{\nust{B}}$. \label{lemmaEtoE}
\end{Lemma}
\begin{proof} Let $_AE$ (respectively, $_B\bar{E}$) be the maximal
$\nu$-stable $A$-module (respectively, $B$-module).  Then by
definition, we have $\nust{A}=\add(_AE)$ and
$\nust{B}=\add(_B\bar{E})$.  The complex
 $F(_AE)$ is isomorphic to a complex $\cpx{\bar{T}_1}$ in
 $\add(\cpx{\bar{T}})$. Since $\nu_AE\simeq {}_AE$, we have
 $\nu_B\cpx{\bar{T}_1}\simeq\cpx{\bar{T}_1}$ in $\Db{B}$.  Hence there
 is a chain map $\eta$ from $\cpx{\bar{T}_1}$ to
 $\nu_B\cpx{\bar{T}_1}$ such that the mapping cone $\con(\eta)$ is
 acyclic. By our assumption, all
 $\bar{T}_1^i$ and $\nu_B\bar{T}_1^i$ with $i\neq 0$ are projective-injective since they are all in $\nust{B}$.
 Hence $\con(\eta)$ splits, and therefore
 $\nu_B\bar{T}_1^0\oplus \bar{Q}_1\simeq\bar{T}_1^0\oplus \bar{Q}_2$ for some $\bar{Q}_1,
 \bar{Q}_2\in\nust{B}$. Hence,
 $\nu_B\bar{T}_1^0\in\add(\bar{T}_1^0\oplus {}_B\bar{E})$. It follows that
 $\nu_B^i\bar{T}_1^0\in\add(\bar{T}_1^0\oplus {}_B\bar{E})$ is projective-injective for all $i\geq 0$.
 Hence $\bar{T}_1^0\in\nust{B}$, and consequently
 $\cpx{\bar{T}_1}$ is in $\Kb{\nust{B}}$. Similarly, we can
 show that $F^{-1}(_B\bar{E})$ is isomorphic to a complex in
 $\Kb{\nust{A}}$ and the lemma is proved.
\end{proof}
The following lemma is useful in the proof of Theorem \ref{Theorem
1}.
\begin{Lemma}
  Let $F:\Db{A}\lra\Db{B}$ and $G:\Db{B}\lra\Db{C}$ be derived
  equivalences, and let $\cpx{P}, \cpx{\bar{P}}, \cpx{Q}$,
  $\cpx{\bar{Q}}$, $\cpx{T}$, and $\cpx{\bar{T}}$ be the tilting complexes associated to $F$,
  $F^{-1}, G$, $G^{-1}$, $GF$, and $F^{-1}G^{-1}$ respectively.
  If the following hold:

    $(1)$ $_AP^{\pm}\in\nust{A}$ and $_B\bar{P}^{\pm}\in\nust{B}$;

    $(2)$ $_BQ^{\pm}\in\nust{B}$ and $_C\bar{Q}^{\pm}\in\nust{C}$,

{\parindent=0pt then } we have $_AT^{\pm}\in\nust{A}$ and
$_C\bar{T}^{\pm}\in\nust{C}$. \label{necessaryCond}
\end{Lemma}
\begin{proof}
   We only need to show that $\bar{T}^{\pm}\in\nust{C}$, the other
   statement follows by symmetry. By definition, $\cpx{\bar{T}}$ is
   isomorphic to $GF(A)\simeq G(\cpx{\bar{P}})$. Since
   $\bar{P}^i\in\nust{B}$ for all $i\neq 0$, by Lemma
   \ref{lemmaEtoE}, $G(\bar{P}^i)$ is isomorphic to a complex
   $\cpx{Y_i}$ in $\Kb{\nust{C}}$ for all $i\neq 0$.
   For $i=0$, the
   complex $G(\bar{P}^0)$ is isomorphic to a complex $\cpx{Y_0}$ in
   $\add(\cpx{\bar{Q}})$. By Lemma \ref{lemmaFormofFX}, the complex
   $G(\cpx{\bar{P}})$ is isomorphic to a complex $\cpx{Z}$ with
   $Z^m=\bigoplus_{i+j=m}Y_i^j$. Since all $Y_i^j$, except $Y_0^0$,
   are in $\nust{C}$, we have $Z^{\pm}\in\nust{C}$. Note that both
   $\cpx{\bar{T}}$ and $\cpx{Z}$ are in $\Kb{\pmodcat{C}}$. The
   complexes $\cpx{\bar{T}}$ and $\cpx{Z}$ are isomorphic in
   $\Kb{\pmodcat{C}}$. Furthermore,  since the complex $\cpx{\bar{T}}$ is a
   radical complex, it follows that $\bar{T}^i$ is a direct summand
   of $Z^i$ for integers $i$, and consequently $\bar{T}^{\pm}\in\nust{C}$.
\end{proof}

Finally, we have the following lemma which is crucial in the proof
of our main result.

\begin{Lemma}
  Let $F:\Db{A}\lra\Db{B}$ be a derived equivalence between two Artin algebras $A$ and $B$, and let
  $\cpx{T}$ be the associated tilting complex of $F$. If
  $_AT^{\pm}\in\nust{A}$, then there is an almost
  $\nu$-stable equivalence $G:\Db{C}\lra\Db{A}$ such that associated
  tilting complex $\cpx{P}$ of $FG$ satisfies that $P^i\in\nust{C}$
  for all $i<0$ and $P^i=0$ for all $i>0$. \label{sufficientCond}
\end{Lemma}
\begin{proof}
Let $_AE$ be the maximal $\nu$-stable $A$-module. Then
$\nust{A}=\add(_AE)$.  Suppose $m$ is the maximal integer such that
$T^m\neq 0$. By a dual statement of \cite[Proposition
3.2]{IdemsTilting}, there is a tilting complex
$\cpx{Q}:=\cpx{R}\oplus {}_AE[-m]$ over $A$, where $\cpx{R}$ is of
the form: $\cpx{R}: 0\lra A\lra R^1\lra \cdots\lra R^m\lra 0$ with
$R^i\in\nust{A}$ for all $i>0$. Let $C$ be the endomorphism algebra
of $\cpx{Q}$, and let $H: \Db{A}\lra\Db{C}$ be a derived equivalence
given by the tilting complex $\cpx{Q}$. It is easy to see that
$H(_AE)\simeq {}_CP[m]$ for some ${}_CP\in\nust{C}$, and $H(A)$ is
isomorphic to a complex $\cpx{S}$: $0\lra S^{-m}\lra\cdots\lra
S^{-1}\lra S^0\lra 0$ with $S^i\in\nust{C}$ for all $i<0$. Let $G$
is a quasi-inverse of $H$. Then $\cpx{S}$ is a tilting complex
associated to $G$. By Lemma \ref{equivCond}, we see that $G$ is
almost $\nu$-stable.

 Now let $\cpx{Y_i}:=H(T^i)$ for each integer $i$. Since $T^{\pm}\in\nust{A}$,
for each integer $i\neq 0$, we have $\cpx{Y_i}\simeq P_i[m]$ for
some $P_i\in\nust{C}$. Moreover, $\cpx{Y_i}=0$ for all $i>m$ since
$T^i=0$ for all $i>m$. The complex $\cpx{Y_0}$ has the property that
$Y_0^i=0$ for all $i>0$ and $Y_0^i\in\nust{C}$ for all $i<0$. By
Lemma \ref{lemmaFormofFX}, the complex $H(\cpx{T})$ is isomorphic to
a complex $\cpx{Z}$ with $Z^t=\bigoplus_{i+j=t}Y_i^j$. It follows
that $Z^t=0$ for all $t>0$ and $Z^t\in\nust{C}$ for all $t<0$. Since
$FG(H(\cpx{T}))\simeq F(\cpx{T})\simeq B\simeq FG(\cpx{P})$ in
$\Db{B}$, the complex $\cpx{Z}$ is isomorphic in $\Db{C}$ to the
tilting complex $\cpx{P}$ associated to $FG$. Since both $\cpx{Z}$
and $\cpx{P}$ are in $\Kb{\pmodcat{C}}$, they are isomorphic in
$\Kb{\pmodcat{C}}$. Since $\cpx{P}$ is a radical complex, the term
$P^i$ is a direct summand of $Z^i$ for all $i$, and consequently
$\cpx{P}$ has the desired property.
\end{proof}
We are now in the position to give a proof of our main result.
\begin{proof}[{\bf Proof\, of\,  Theorem \ref{Theorem 1}}]
   Assume that $F$ is an iterated almost $\nu$-stable derived
   equivalence. Let $F\simeq F_1F_2\cdots F_m$ be a composition such that
   $F_i$ or $F_i^{-1}$ is an almost $\nu$-stable derived
   equivalence. Then by Lemma \ref{necessaryCond}, we have
   $\add(_AT^{\pm})=\add(\nu_AT^{\pm})$ and
   $\add(_B\bar{T}^{\pm})=\add(\nu_B\bar{T}^{\pm})$. Conversely,
   assume that $\add(_AT^{\pm})=\add(\nu_AT^{\pm})$ and
   $\add(_B\bar{T}^{\pm})=\add(\nu_B\bar{T}^{\pm})$. By Lemma
   \ref{sufficientCond}, there is an almost $\nu$-stable derived
   equivalence $G: \Db{C}\lra\Db{A}$ such that the
   tilting complex $\cpx{P}$ associated to $FG$ has the property that $P^i=0$ for all
   $i>0$ and $P^i\in\nust{C}$ for all $i<0$. By Lemma
   \ref{equivCond}, we have
   $\add(\bigoplus_{i<0}P^i)=\add(\bigoplus_{i<0}\nu_CP^i)$. Let $\cpx{\bar{P}}$ be
   the tilting complex associated to $G^{-1}F^{-1}$. By Lemma
   \ref{necessaryCond}, we have
   $\add(_B\bar{P}^{\pm})=\add(\nu_B\bar{P}^{\pm})$. Since $P^i=0$ for
   all $i>0$, by \cite[Lemma 2.1]{HuXi3}, we have $\bar{P}^i=0$ for
   all $i<0$. Hence $FG$ is an almost $\nu$-stable derived
   equivalence. Thus, $F\simeq (FG) G^{-1}$ is an iterated almost
   $\nu$-stable derived equivalence.
\end{proof}
  {\it Remark:} (1) Theorem \ref{Theorem 1} gives us a method to
determine whether a derived equivalence is iterated almost
$\nu$-stable or not by checking the terms of the involved tilting
complexes.

(2) Let $P$ be a projective $A$-module. The condition
$\add(_AP)=\add(\nu_AP)$ is equivalent to say that $P$ is
projective-injective and $\add(\top(P))=\add(\soc(P))$.

(3) The proof of Corollary \ref{Corollary 2} follows immediately
from Theorem \ref{Theorem 1} and \cite[Theorem 5.3]{LiuXi3}. But
Corollary \ref{Corollary 2} generalizes \cite[Theorem 5.3]{LiuXi3}.
This gives a new sufficient condition for a derived equivalence to
induce a stable equivalence of Morita type.

\medskip
As an application of Theorem \ref{Theorem 1}, we have the following
corollary.
\begin{Corollary}
 Let $F:\Db{A}\lra \Db{B}$ be a derived equivalence between two
Artin algebras $A$ and $B$, and let $\cpx{T}$ and $\cpx{\bar{T}}$ be
the tilting complexes associated to $F$ and $F^{-1}$, respectively.
If $\add(_AT^{\pm})=\add(\nu_AT^{\pm})$ and
$\add(_B\bar{T}^{\pm})=\add(\nu_B\bar{T}^{\pm})$, then the following
hold:

$(1)$ $\findim(A)=\findim(B)$, and $\gldim(A)=\gldim(B)$;

$(2)$ $\repdim(A)=\repdim(B)$;

$(3)$ $\domdim(A)=\domdim(B)$,

{\parindent=0pt where} $\findim, \gldim, \repdim$ and $\domdim$
stand for finitistic dimension, global dimension, representation
dimension and dominant dimension, respectively.
\end{Corollary}
\begin{proof}
The corollary follows from \cite[Corollary 1.2]{HuXi3} and Theorem
\ref{Theorem 1}.
\end{proof}
Now we work with finite-dimensional algebras over a field. In this
case, we get several other characterizations of iterated almost
$\nu$-stable derived equivalences, which is the following theorem.

\begin{Theorem}
   Let $F:\Db{A}\lra\Db{B}$ be a derived equivalence between two
   finite-dimensional basic algebras $A$ and $B$ over a field, and let
   $\cpx{T}$ and $\cpx{\bar{T}}$ be the tilting complexes associated
   to $F$ and $F^{-1}$, respectively. Then the following are
   equivalent:

 $(1)$ The functor $F$ is an iterated almost $\nu$-stable derived
   equivalence.

 $(2)$ $\add(\nu_AT^{\pm})=\add({}_AT^{\pm})$ and $\add(\nu_B\bar{T}^{\pm})=\add({}_B\bar{T}^{\pm})$.

 $(3)$ $T^{\pm}\in\nust{A}$ and $\bar{T}^{\pm}\in\nust{B}$.

 $(4)$ For each indecomposable projective $A$-module $P\not\in\nust{A}$,
   the image $F(\top (P))$ is isomorphic in $\Db{B}$ to a simple
   $B$-module.

 $(5)$ For each indecomposable projective $A$-module $P\not\in\nust{A}$, the
   following conditions are satisfied:

   {\parindent=30pt $(a)$ $P\not\in\add({}_AT^{\pm})$; }

   {\parindent=30pt $(b)$ the multiplicity of $P$ as a direct summand of $_AT^0$ is $1$.}
   \label{TheoremOtherChar}
\end{Theorem}
\begin{proof}
  It follows from Theorem \ref{Theorem 1} and Lemma \ref{equivCond}
  that the statements $(1)$, $(2)$ and $(3)$ are equivalent. Note
  that for any simple module $S$ over a basic algebra $\Lambda$, the dimension
  of $S$ as an $\End_{\Lambda}(S)$-space is $1$. In this proof, let $_AE$ and $_B\bar{E}$ be the maximal
  $\nu$-stable $A$-module and $B$-module, respectively.

  $(4)\Rightarrow (5)$  For each indecomposable projective
  $A$-module $P$ not in $\nust{A}$, since $F(\top(P))$ is isomorphic
  in $\Db{B}$ to a simple $B$-module, we have $\Hom_{\Db{A}}(\cpx{T},
  \top(P)[i])=0$ for all $i\neq 0$, and $$\Hom_{\Db{A}}(\cpx{T},
  \top(P))\simeq \Hom_B(B, F(\top(P)))\simeq F(\top(P))$$ is one-dimensional over the division ring
  $\End_A(\top(P))$. Note that $\cpx{T}$ is a radical complex. It
  follows that $P$ is not a direct summand of $T^{\pm}$ and the
  multiplicity of $P$ as a direct summand of $T^0$ is $1$.

  $(5)\Rightarrow (4)$  By condition $(a)$, we see that $\Hom_{\Db{A}}(\cpx{T},
  \top(P)[i])=0$ for all $i\neq 0$. Hence $F(\top(P))$ is isomorphic
  to an indecomposable $B$-module $X$. By condition $(b)$, we can assume that
  $\cpx{T_P}$ is the only indecomposable direct summand of $\cpx{T}$
  such that $P$ is a direct summand of its degree zero term. Suppose that $\bar{P}$ is the indecomposable projective $B$-module
  corresponding to the direct summand $\cpx{T_P}$. Then
  $$\Hom_B(B, X)\simeq \Hom_{\Db{A}}(\cpx{T}, \top(P))\simeq \Hom_{\Db{A}}(\cpx{T_P}, \top(P))\simeq \Hom_B(\bar{P}, X).$$
  This implies that $X$ only contains $\top(\bar{P})$ as composition
  factors. If $X$ is not a simple $B$-module, then there is a nonzero map
  $X\lra \soc(X)\lra X$ in $\End_B(X)$ which
  is not an isomorphism. This contradicts to the fact that $\End_B(X)\simeq\End_A(\top(P))$ is a
  division ring. Hence $X\simeq F(\top(P))$ is a simple $B$-module.

  $(3)\Rightarrow (4)$ By
  definition, we have $\add({}_AE)=\nust{A}$ and
  $\add({}_B\bar{E})=\nust{B}$. Let $P$ be an indecomposable
  projective $A$-module not in $\nust{A}$. Then it is clear that $\Hom_{\Db{A}}(\cpx{T},
  \top(P)[i])=0$ for all $i\neq 0$ since $T^{\pm}\in\nust{A}$, and consequently $F(\top(P))$ is
  isomorphic in $\Db{B}$ to a $B$-module $X$. By Lemma
  \ref{lemmaEtoE}, the complex $F^{-1}({}_B\bar{E})$ is isomorphic
  in $\Db{A}$ to a complex $\cpx{E}$ in $\Kb{\nust{A}}$. Hence
  $$\Hom_B({}_B\bar{E}, X)\simeq \Hom_{\Db{A}}(F^{-1}({}_B\bar{E}),
  \top(P))\simeq \Hom_{\Kb{A}}(\cpx{E}, \top(P))=0.$$ If $_BX$ is not
  simple, then there is a short exact sequence $0\lra \bar{U}\lra X\lra \bar{V}\lra
  0$ in $B$-module with $\bar{U}, \bar{V}$ non-zero. Applying $\Hom_B({}_B\bar{E},
  -)$, we get that $\Hom_B(_B\bar{E}, \bar{U})=0=\Hom_B({}_B\bar{E}, \bar{V})$,
  and consequently $\Hom_{\Db{B}}(\cpx{\bar{T}}, \bar{U}[i])=0=\Hom_{\Db{B}}(\cpx{\bar{T}},
  \bar{V}[i])$ for all $i\neq 0$ since $\bar{T}^{\pm}\in\nust{B}$. Hence $F^{-1}(\bar{U})$ and
  $F^{-1}(\bar{V})$ are isomorphic to $A$-modules $U$ and $V$,
  respectively. Thus, we get a distinguished triangle $$U\lra \top(P)\lra V\lra
  U[1]$$  in $\Db{A}$ by applying $F^{-1}$ to the distinguished
  triangle $\bar{U}\lra X\lra \bar{V}\lra \bar{U}[1]$. Applying $\Hom_{\Db{A}}(A,
  -)$ to the above triangle, we get an exact sequence $0\lra U\lra \top(P)\lra V\lra
  0$ with non-zero $A$-modules $U$ and $V$. This contradicts to the
  fact that $\top(P)$ is a simple $A$-module. Hence $F(\top(P))\simeq X$ is a simple
  $B$-module.

  $(4)\Rightarrow (3)$ For each indecomposable projective $A$-module
  $P$ not in $\nust{A}$, since $F(\top(P))$ is isomorphic $\Db{B}$ to a simple $B$-module, we have
  $\Hom_{\Db{A}}(\cpx{T},\top(P)[i])=0$ for all $i\neq 0$. Together
  with the isomorphism
  $$\Hom_{\Db{A}}(\cpx{T}, \top(P)[i])\simeq \Hom_{\Kb{A}}(\cpx{T}, \top(P)[i])\simeq \Hom_A(T^i,
  \top(P)), $$
  we get $\Hom_{A}(T^i, \top(P))=0$ for all $i\neq 0$ and for all
  indecomposable projective $A$-module $P$ not in $\nust{A}$. Hence
  $T^{i}\in\nust{A}$ for all $i\neq 0$, that is,
  $T^{\pm}\in\nust{A}$. Now let $_AQ$ be a projective $A$-module
  such that $_AA\simeq {}_AE\oplus {}_AQ$. It follows by assumption that
  $F(\top(Q))$ is a semi-simple $B$-module. Suppose that $\bar{Q}$
  is a projective cover of $F(\top(Q))$, and suppose that $_BB\simeq \bar{Q}\oplus
  W$. Since $\cpx{\bar{T}}$ is a radical complex in $\pmodcat{B}$,
  we have
  $$\begin{array}{rl}
  \Hom_B(\bar{T}^i, \top(\bar{Q})) & \simeq \Hom_{\Kb{B}}(\cpx{\bar{T}},
  \top(\bar{Q})[i])\\
  & \simeq \Hom_{\Db{B}}(\cpx{\bar{T}},
  \top(\bar{Q})[i])\\
  &\simeq \Hom_{\Db{A}}(A, \top(Q)[i])=0\end{array}$$
  for all $i\neq 0$. Hence $\bar{T}^{\pm}\in\add(_BW)$. It remains
  to show $_BW\in\nust{B}$. Note that $\Hom_B({}_BW, \top(\bar{Q})[i])=0$ for all integers $i$.
  It follows that $\Hom_{\Db{A}}(F^{-1}(_BW), \top({Q})[i])=0$ for
  all integers $i$. Let $\cpx{L}$ be a radical complex in
  $\Kb{\pmodcat{A}}$ such that $F^{-1}(_BW)\simeq \cpx{{L}}$. Then $\Hom_A({L}^i, \top({Q}))\simeq \Hom_{\Db{A}}(\cpx{{L}},
  \top({Q})[i])=0$ for all integers $i$. Hence ${L}^i\in\add(_AE)$
  for all integers $i$. Using the same proof as the proof of
  \cite[Theorem 2.1]{AlRickard}, we can show that $\nu_B^iW$ is
  a projective $B$-module for all $i\geq 0$. It follows that $_BW\simeq
  \nu_B^kW$ for some $k>0$. Hence $_BW$ is projective-injective and
  $\nu_B^iW$ is projective-injective for $i>0$, and consequently
  $_BW\in\nust{B}$. This finishes the proof.
\end{proof}
{\it Remark:} (1) By Theorem \ref{TheoremOtherChar} (5), we see that
if we consider finite-dimensional algebras over a field, then we can
determine whether a derived equivalence $F$ is iterated almost
$\nu$-stable or not by checking the terms of the tilting complex
associated to $F$, and we do not need to check the terms of the
tilting complex associated to $F^{-1}$, which is needed in Theorem
\ref{Theorem 1}.

(2) It is interesting to know whether Theorem \ref{TheoremOtherChar}
holds for general Artin algebras. Note that the only problem is the
step ``$(4)\Rightarrow (3)$", where the method in the proof of
\cite[Theorem 2.1]{AlRickard} does not work for general Artin
algebras.

\medskip
As a consequence, we have the following corollary.
\begin{Corollary}
   Let $F:\Db{A}\lra \Db{B}$ be a derived equivalence between two
   finite-dimensional basic algebras over a field. If one of the
   equivalent conditions in {\rm Theorem \ref{TheoremOtherChar}} is
   satisfied, then the algebras $A$ and $B$ are stably equivalent of
   Morita type. \label{CorollaryDeriveStable}
\end{Corollary}
\begin{proof}
  This follows from Theorem \ref{TheoremOtherChar}, and \cite[Theorem
  5.3]{HuXi3}.
\end{proof}

We end this section by using a simple example to illustrate Theorem
\ref{TheoremOtherChar} and Corollary \ref{CorollaryDeriveStable}.

\smallskip
 {\parindent=0pt\bf Example}: Let $k$ be a field, and let $A$ and $B$ be
finite-dimensional $k$-algebras given by quivers with relations in
Fig. 1 and Fig. 2, respectively.
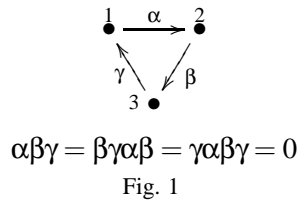
\begin{center}
\begin{tabular}{ccc}
\multirow{3}{*}{$\xymatrix@R=2mm@C=2mm{
 \bullet\ar[rr]^{\alpha}^(0){1}^(1){2} && \bullet\ar[ldd]^{\beta}\\
      &&\\
      &\bullet\ar[luu]^(.15){3}^{\gamma} &
}$} & \hspace{1cm} & $
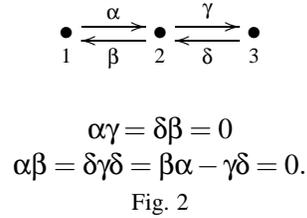
\xymatrix{
  \bullet\ar@<2.5pt>[r]^{\alpha} &
  \bullet\ar@<2.5pt>[l]^(1){1}^(0){2}^{\beta}\ar@<2.5pt>[r]^{\gamma}
  &\bullet\ar@<2.5pt>[l]^(0){3}^{\delta}
}$\\
&&\\
&& $\alpha\gamma=\delta\beta=0$\\
$\alpha\beta\gamma=\beta\gamma\alpha\beta=\gamma\alpha\beta\gamma=0$
&&
$\alpha\beta=\delta\gamma\delta=\beta\alpha-\gamma\delta=0.$\\
{\footnotesize Fig. 1 } && {\footnotesize Fig. 2 }\\
\end{tabular}
\end{center}
Let $P(i)$ denote the indecomposable projective $A$-module
corresponding to the vertex $i$. Then there is a tilting complex of
$A$-modules
$$\cpx{T}: \quad 0 \lra P(2)\oplus P(2)\oplus P(3)
\lraf{[f,0,0]^T} P(1)\lra 0$$ with $P(1)$ in degree zero. One can
check that $\End_{\Kb{\pmodcat{A}}}(\cpx{T})$ is isomorphic to $B$,
and that $\nust{A}=\add(P(2)\oplus P(3))$. Hence the tilting complex
satisfies the condition (5) in Theorem \ref{TheoremOtherChar}.
Therefore, the complex $\cpx{T}$ induces an iterated almost
$\nu$-stable derived equivalence (actually even an almost
$\nu$-stable derived equivalence) between $A$ and $B$. By Corollary
\ref{CorollaryDeriveStable}, the algebras $A$ and $B$ are stably
equivalent of Morita type.

\section{The stable equivalence functor}
In this section, we will give a description of the stable
equivalence functor induced by an iterated almost $\nu$-stable
derived equivalence.

Let $A$ be an Artin algebra, and let $_AE$ be a maximal $\nu$-stable
$A$-module. Then by definition $\nust{A}=\add(_AE)$.  We use
$\nustcat{A}$
 to denote the quotient category of $\modcat{A}$ modulo morphisms
 factorizing through modules in $\nust{A}$. The Hom-space in
 $\nustcat{A}$ is denoted by $\nustHom_A(-, -)$. For a morphism $f$
 in $\modcat{A}$, its image in $\nustcat{A}$ under the canonical
 functor from $\modcat{A}$ to $\nustcat{A}$ is denoted by
 $\nustmor{f}$.
  The category $\Kb{\nust{A}}$ is a clearly thick subcategory (that is, a
triangulated full subcategory closed under taking direct summands)
of $\Db{A}$. Let $\Db{A}/\Kb{\nust{A}}$ be the Verdier quotient
category, then we have a canonical additive functor
$$\Sigma': \modcat{A}\lra \Db{A}/\Kb{\nust{A}}$$
obtained by composing the natural embedding from $\modcat{A}$ to
$\Db{A}$ and the quotient functor from
$\Db{A}\lra\Db{A}/\Kb{\nust{A}}$. For the definition and basic
properties of Verdier quotient, we refer to \cite[Chapter
2]{neumann}. Since $\Sigma'(_AE)$ is clearly isomorphic to zero
object in $\Db{A}/\Kb{\nust{A}}$, the functor $\Sigma'$ induces an
additive functor
$$\Sigma: \nustcat{A}\lra \Db{A}/\Kb{\nust{A}}.$$

Keeping this notation, we have a proposition, which can be viewed as
a generalization of a well-known result of Rickard \cite[Theorem
2.1]{RickDstable}
\begin{Proposition}
   The functor $$\Sigma: \nustcat{A}\lra\Db{A}/\Kb{\nust{A}}$$ is
   fully faithful. Moreover, the functor $\Sigma$ is an equivalence if
   and only if $A$ is self-injective. \label{PropositionEmbedding}
\end{Proposition}
\begin{proof}
   A morphism $\cpx{X}\lra \cpx{Y}$ in $\Db{A}/\Kb{\nust{A}}$ is
   denoted by a fraction $s^{-1}a:
   \cpx{X}\llaf{s}\cpx{Z}\lraf{a}\cpx{Y}$, where $a$ and $s$ are
   morphisms in $\Db{A}$, and if $\cpx{Z}\rlaf{s}\cpx{X}\lra
   \cpx{U}\lra\cpx{Z}[1]$ is a distinguished triangle in $\Db{A}$,
   then $\cpx{U}\in\Kb{\nust{A}}$. A morphism $s'$ in $\Db{A}$ with this
   property will be denoted by $\rlaf{s'}$. Two morphisms
   $\cpx{X}\llaf{s}\cpx{U}\lraf{a}\cpx{Y}$ and
   $\cpx{X}\llaf{r}\cpx{V}\lraf{b}\cpx{Y}$ are equal if and only if
   there are morphisms $\cpx{W}\rlaf{t}\cpx{U}$ and
   $\cpx{W}\rlaf{h}\cpx{V}$ such that $ts=hr$ and $ta=hb$. An
   isomorphism from $X$ to $Y$ is of the form
   $X\llaf{s}\cpx{U}\rlaf{t}Y$.

   First, we show that $\Sigma$ is a full functor. For this purpose, it
   suffices to show that $\Sigma'$ is a full functor. Now let $f: X\lra
   Y$ be a morphism in $\modcat{A}$. Then $\Sigma'(f)$ is the morphism
   $X\llaf{1_X}X\lraf{f}Y$. We need to show that each morphism from
   $X$ to $Y$ in $\Db{A}/\Kb{\nust{A}}$ is of this form. Let
   $X\llaf{s}\cpx{U}\lraf{a}Y$ be a morphism in
   $\Db{A}/\Kb{\nust{A}}$. By definition, there is a distinguished triangle
   $\cpx{U}\lraf{s} X\lraf{g} \cpx{E}\lra \cpx{U}[1]$ in $\Db{A}$ with
   $\cpx{E}\in\Kb{\nust{A}}$. Consider the distinguished triangle
   in $\Db{A}$
   $$\sigma_{\geq
   0}\cpx{E}\lraf{\alpha}\cpx{E}\lraf{\beta}\sigma_{<0}\cpx{E}\lra (\sigma_{\geq 0}\cpx{E})[1].$$
Since $\cpx{E}$ is clearly in $\Kb{\imodcat{A}}$, we have
$\Hom_{\Db{A}}(X, \sigma_{<0}\cpx{E})\simeq \Hom_{\Kb{A}}(X,
\sigma_{<0}\cpx{E})=0$. It follows that $g\beta=0$, and therefore
$g$ factorizes through $\alpha$. Hence we can form the following
commutative diagram in $\Db{A}$ with rows being distinguished
triangles.
$$\xymatrix{
\cpx{V}\ar@{=>}[r]^h\ar@{=>}[d]^{r} & X\ar[r]\ar@{=}[d] &
\sigma_{\geq 0}\cpx{E}\ar[r]^{w}\ar[d]^{\alpha}
&\cpx{V}[1]\ar@{=>}[d]^{r[1]}\\
 \cpx{U}\ar@{=>}[r]^s & X\ar[r]^{g} & \cpx{E}\ar[r]
&\cpx{U}[1]. }$$ Since $\Hom_{\Db{A}}((\sigma_{\geq 0}\cpx{E})[-1],
Y)\simeq \Hom_{\Kb{A}}((\sigma_{\geq 0}\cpx{E})[-1], Y)=0$, the
morphism $(w[-1])ra=0$, and hence there is some morphism $f: X\lra
Y$ in $\Db{A}$ such that $ra=hf$. Then we have the following
commutative diagram in $\Db{A}$
$$\xymatrix@R=4mm @C=4mm{
  & \cpx{V}\ar@{=>}[rd]^{h}\ar@{=>}[ld]_{r} \\
  \cpx{U}\ar@{=>}[d]_{s}\ar[rrd]_(.7){a} && X\ar@{=>}[lld]_(,3){1_X}\ar[d]^{f}\\\
  X && Y,
}$$ which means that the morphisms  $X\llaf{s}\cpx{U}\lraf{a}Y$ and
$X\llaf{1_X}X\lraf{f}Y$ in $\Db{A}/\Kb{\nust{A}}$ are equal. Since
the embedding of $\modcat{A}$ into $\Db{A}$ is fully faithful, the
morphism $f$ is given by a morphism in $\modcat{A}$. Hence the
functor $\Sigma'$ is full, and therefore $\Sigma$ is a full functor.

Suppose that $f: X\lra Y$ is a morphism in $\modcat{A}$ such that
$\Sigma'(f)=0$. That is, the morphisms $X\llaf{1_X}X\lraf{0}Y$ and
$X\llaf{1_X}X\lraf{f}Y$ are equal in $\Db{A}/\Kb{\nust{A}}$. Then
there is a morphism $\cpx{W}\rlaf{s} X$ such that $sf=0$ in
$\Db{A}$. Embedding $s$ into a distinguished triangle in $\Db{A}$,
we see that $f$ factorizes in $\Db{A}$ through a complex in
$\Kb{\nust{A}}$, and therefore it follows easily that $f$ factorizes
in $\modcat{A}$ through an $A$-module $\nust{A}$. Hence the functor
$\Sigma$ is faithful.

If $A$ is self-injective, then $\nust{A}=\pmodcat{A}$ and the
equivalence was proved by Rickard \cite[Theorem 2.1]{RickDstable}.
If $A$ is not self-injective, then there is a projective $A$-module
$P$ not in $\nust{A}$. Suppose that $\Sigma$ is an equivalence. Then
there is some $A$-module $X$ such that $X\simeq P[-1]$ in
$\Db{A}/\Kb{\nust{A}}$. That is, there is an isomorphism
$X\llaf{s}\cpx{U}\rlaf{t}P[-1]$ in $\Db{A}/\Kb{\nust{A}}$. Then by
Octahedral Axiom, we can form the following commutative diagram in
$\Db{A}$
$$\xymatrix{
  \cpx{E_1}\ar[r]\ar@{=}[d] & \cpx{U}\ar@{=>}[r]^{t}\ar@{=>}[d]^{s} & P[-1]\ar[r]\ar[d]^{h} &\cpx{E_1}[1]\ar@{=}[d]\\
  \cpx{E_1}\ar[r]^{g} & X\ar[r]^{t}\ar[d] & \con(g)\ar[r]\ar[d] &\cpx{E_1}[1], \\
  & \cpx{E_2}\ar@{=}[r] &\cpx{E_2}
}$$ where $\cpx{E_1}$ and $\cpx{E_2}$ are in $\Kb{\nust{A}}$, and
$\con(g)$ is the mapping cone of $g$. From the vertical
distinguished triangle on the right side, we see that the mapping
cone $\con(h)$ of $h$ is isomorphic in $\Db{A}$ to a complex
$\cpx{E_2}$ in $\Kb{\nust{A}}$. All the terms of $\con(h)$ in
non-zero degrees are in $\nust{A}$ and $P\oplus X$ is a direct
summand of the $0$-degree term of $\con(h)$. Hence $P$ is isomorphic
to a complex in $\Kb{\nust{A}}$ which is impossible since $P$ is
projective and is not in $\nust{A}$. This finishes the proof.
\end{proof}

{\it Remark:} In the above proposition, suppose that $P$ is a
projective-injective $A$-module, if we replace $\nustcat{A}$ by the
quotient category of $\modcat{A}$ modulo morphisms factorizing
through modules in $\add(P)$, and replace $\Db{A}/\Kb{\nust{A}}$ by
$\Db{A}/\Kb{\add(P)}$, then the proof of Proposition
\ref{PropositionEmbedding} actually can be used to show that in this
case the functor $\Sigma$ is also fully faithful.

\medskip

Now for each iterated almost $\nu$-stable derived equivalence
$F:\Db{A}\lra\Db{B}$. By Lemma \ref{lemmaEtoE}, we see that $F$
induces an equivalence between the triangulated categories
$\Db{A}/\Kb{\nust{A}}$ and $\Db{B}/\Kb{\nust{B}}$. We also denote
this equivalence by $F$. In the following, we will see that there is
an equivalence $\phi_F: \nustcat{A}\lra \nustcat{B}$ such that the
diagram
$$\xymatrix{
  \nustcat{A}\ar[r]^(.35){\Sigma}\ar[d]^{\phi_F} & \Db{A}/\Kb{\nust{A}}\ar[d]^{F}\\
  \nustcat{B}\ar[r]^(.35){\Sigma} & \Db{B}/\Kb{\nust{B}}
 }$$
of additive functors is commutative up to isomorphism. Moreover, the
functor $\phi_F$ also induces an equivalence between the stable
module categories $\stmodcat{A}$ and $\stmodcat{B}$.

\medskip
Before we give the construction of $\phi_F$, we give the following
lemma, which generalizes \cite[Lemma 2.2]{HuXi3} and will be
 used in the construction of $\phi_F$.
\begin{Lemma}
   Let $A$ be an arbitrary ring, and let $A\mbox{\rm -Mod}$ be the
category of all left (not necessarily finitely generated)
$A$-modules. Suppose $\cpx{X}$ is a complex over $A\mbox{\rm -Mod}$
bounded above and $\cpx{Y}$ is a complex over $A\mbox{\rm -Mod}$
bounded below. If there is an integer $m$ such that $X^i$ is
projective for all $i>m$ and $Y^j$ is injective for all $j<m$, then
$\theta_{\cpx{X},\cpx{Y}}: \Hom_{\K{A\mbox{\rm
-Mod}}}(\cpx{X},\cpx{Y})\rightarrow \Hom_{\D{A\mbox{\rm
-Mod}}}(\cpx{X},\cpx{Y})$ induced by the localization functor
$\theta: \K{A\mbox{\rm -Mod}}$ $\rightarrow\D{A\mbox{\rm -Mod}}$ is
an isomorphism. \label{LemmaKDiso}
\end{Lemma}
\begin{proof}
Without loss of generality, we can assume that $m=0$. For
simplicity, we write $\mathscr{K}$ for $\K{A\mbox{\rm -Mod}}$ and
$\mathscr{D}$ for $\D{A\mbox{\rm -Mod}}$. Also, the Hom-spaces
$\Hom_{\mathscr{K}}(-, -)$ and $\Hom_{\mathscr{D}}(-, -)$ will be
denoted by $_{\mathscr{K}}(-, -)$ and $_{\mathscr{D}}(-, -)$,
respectively.

First, we show that, for each $A$-module $Z$, the induced map
$$\theta_{\cpx{X}, Z[1]}: {}_{\mathscr{K}}(\cpx{X}, Z[1])\lra {}_{\mathscr{D}}(\cpx{X}, Z[1])$$
is monic. Indeed, applying $_{\mathscr{K}}(-, Z[1])$ and
$_{\mathscr{D}}(-, Z[1])$ to the distinguished triangle
$$\sigma_{\geq 0}\cpx{X}\lra \cpx{X}\lra \sigma_{<0}\cpx{X}\lra (\sigma_{\geq 0}\cpx{X})[1],$$
we get a commutative diagram with exact rows.
$$\xymatrix{
  {}_{\mathscr{K}}(\sigma_{\geq 0}\cpx{X}, Z)\ar[r]\ar[d]^{\theta_{\sigma_{\geq 0}\cpx{X}, Z}}  & {}_{\mathscr{K}}(\sigma_{<0}\cpx{X},
  Z[1])\ar[r]\ar[d]^{\theta_{\sigma_{<0}\cpx{X},
  Z[1]}} &  {}_{\mathscr{K}}(\cpx{X}, Z[1])\ar[r]\ar[d]^{\theta_{\cpx{X}, Z[1]}} &  {}_{\mathscr{K}}(\sigma_{\geq 0}\cpx{X},
  Z[1])\ar[d]^{\theta_{\sigma_{\geq 0}\cpx{X}, Z[1]}}\\
    {}_{\mathscr{D}}(\sigma_{\geq 0}\cpx{X}, Z)\ar[r]  & {}_{\mathscr{D}}(\sigma_{<0}\cpx{X},
  Z[1])\ar[r] &  {}_{\mathscr{D}}(\cpx{X}, Z[1])\ar[r] &  {}_{\mathscr{D}}(\sigma_{\geq 0}\cpx{X}, Z[1])
}$$ By \cite[Lemma 2.2]{HuXi3}, the maps $\theta_{\sigma_{\geq
0}\cpx{X}, Z}$ and $\theta_{\sigma_{<0}\cpx{X},Z[1]}$ are
isomorphisms. Since ${}_{\mathscr{K}}(\sigma_{\geq
0}\cpx{X},Z[1])=0$, the map $\theta_{\sigma_{\geq 0}\cpx{X}, Z[1]}$
is clearly monic. Thus, by the Five Lemma (see, for example
\cite[p.13]{Weibel}), the map $\theta_{\cpx{X},Z[1]}$ is monic.

Next, we show that the map
$$\theta_{\cpx{X}, (\sigma_{\geq 0}\cpx{Y})[1]}: {}_{\mathscr{K}}(\cpx{X}, (\sigma_{\geq 0}\cpx{Y})[1])\lra {}_{\mathscr{D}}(\cpx{X}, (\sigma_{\geq 0}\cpx{Y})[1]) $$
is monic. Indeed, applying ${}_{\mathscr{K}}(\cpx{X}, -)$ and
${}_{\mathscr{D}}(\cpx{X}, -)$ to the distinguished triangle
$$Y^0\lra (\sigma_{>0}\cpx{Y})[1]\lra (\sigma_{\geq 0}\cpx{Y})[1]\lra Y^0[1],$$
we get a commutative diagram with exact rows.
$$\xymatrix{
  {}_{\mathscr{K}}(\cpx{X}, Y^0)\ar[r]\ar[d]^{\theta_{\sigma_{\cpx{X}, Y^0}}}  & {}_{\mathscr{K}}(\cpx{X},
  (\sigma_{>0}\cpx{Y})[1])\ar[r]\ar[d]^{\theta_{\cpx{X},
  (\sigma_{>0}\cpx{Y})[1]}} &  {}_{\mathscr{K}}(\cpx{X}, (\sigma_{\geq 0}\cpx{Y})[1])\ar[r]\ar[d]^{\theta_{\cpx{X}, (\sigma_{\geq 0}\cpx{Y})[1]}} &  {}_{\mathscr{K}}(\cpx{X},
  Y^0[1])\ar[d]^{\theta_{\cpx{X}, Y^0[1]}}\\
  {}_{\mathscr{D}}(\cpx{X}, Y^0)\ar[r] & {}_{\mathscr{D}}(\cpx{X},
  (\sigma_{<0}\cpx{Y})[1])\ar[r] &  {}_{\mathscr{D}}(\cpx{X}, (\sigma_{\geq 0}\cpx{Y})[1])\ar[r] &  {}_{\mathscr{D}}(\cpx{X},
  Y^0[1])
}$$ Again by \cite[Lemma 2.2]{HuXi3}, the left two vertical maps are
isomorphisms. By the above discussion, we see that $\theta_{\cpx{X},
Y^0[1]}$ is monic. So, by the Five Lemma again, the map
$\theta_{\cpx{X}, (\sigma_{\geq 0}\cpx{Y})[1]}$ is monic.

Finally, applying ${}_{\mathscr{K}}(\cpx{X}, -)$ and
${}_{\mathscr{D}}(\cpx{X}, -)$ to the distinguished triangle
$$(\sigma_{<0}\cpx{Y})[-1]\lra \sigma_{\geq 0}\cpx{Y}\lra \cpx{Y} \lra \sigma_{<0}\cpx{Y},$$
we get a commutative diagram
$$\xymatrix@C=5mm{
  {}_{\mathscr{K}}(\cpx{X}, (\sigma_{<0}\cpx{Y})[-1])\ar[r]\ar[d]^{\theta_{\cpx{X}, (\sigma_{<0}\cpx{Y})[-1]}} & {}_{\mathscr{K}}(\cpx{X},
  \sigma_{\geq 0}\cpx{Y})\ar[r]\ar[d]^{\theta_{\cpx{X},
  \sigma_{\geq 0}\cpx{Y}}} &  {}_{\mathscr{K}}(\cpx{X}, \cpx{Y})\ar[r]\ar[d]^{\theta_{\cpx{X}, \cpx{Y}}} &  {}_{\mathscr{K}}(\cpx{X},
  \sigma_{<0}\cpx{Y})\ar[d]^{\theta_{\cpx{X}, \sigma_{<0}\cpx{Y}}} \ar[r] & {}_{\mathscr{K}}(\cpx{X},
  (\sigma_{\geq 0}\cpx{Y})[1])\ar[d]^{\theta_{\cpx{X},
  (\sigma_{\geq 0}\cpx{Y})[1]}}\\
  {}_{\mathscr{D}}(\cpx{X}, (\sigma_{<0}\cpx{Y})[-1])\ar[r] & {}_{\mathscr{D}}(\cpx{X},
  \sigma_{\geq 0}\cpx{Y})\ar[r] &  {}_{\mathscr{D}}(\cpx{X}, \cpx{Y})\ar[r] &  {}_{\mathscr{D}}(\cpx{X},
  \sigma_{<0}\cpx{Y}) \ar[r] & {}_{\mathscr{D}}(\cpx{X},
  (\sigma_{\geq 0}\cpx{Y})[1])\\
}$$ By assumption, the complex $\sigma_{<0}\cpx{Y}$ is a bounded
complex of injective $A$-modules. So, the maps $\theta_{\cpx{X},
(\sigma_{<0}\cpx{Y})[-1]}$ and $\theta_{\cpx{X},
\sigma_{<0}\cpx{Y}}$ are isomorphisms. By \cite[Lemmma 2.2]{HuXi3},
the map $\theta_{\cpx{X},\sigma_{\geq 0}\cpx{Y}}$ is an isomorphism.
We have already proved that the map $\theta_{\cpx{X},(\sigma_{\geq
0}\cpx{Y})[1]}$ is monic. Then by applying the Five Lemma again, the
proof is completed.
\end{proof}

Now we fix some notations for the rest of this section. Let
$F:\Db{A}\lra\Db{B}$ be an iterated almost $\nu$-stable derived
equivalences between two Artin algebras $A$ and $B$, and let $G$ be
a quasi-inverse of $F$. Let $\cpx{T}$ and $\cpx{\bar{T}}$ be the
tilting complexes associated to $F$ and $G$, respectively.
Then by Theorem \ref{Theorem 1} and Lemma \ref{equivCond}, the terms
of $\cpx{T}$ in non-zero degrees are all in $\nust{A}$, and the
terms of $\cpx{\bar{T}}$ in non-zero degrees are all in $\nust{B}$.

Keeping these notations, we have the following lemma.
\begin{Lemma}
  For each $A$-module $X$, the complex $F(X)$ is isomorphic in
  $\Db{B}$ to a radical complex $\cpx{\bar{T}_X}$  with
  $\bar{T}_X^{\pm}\in\nust{B}$.
  Moreover, the complex $\cpx{\bar{T}_X}$ of this form is unique up to
  isomorphism in $\Cb{B}$. In particular, if $X$ is a projective (respectively, injective)
  module, then $\cpx{\bar{T}_X}$ is isomorphic in $\Cb{B}$ to a
  complex in $\add(\cpx{\bar{T}})$ (respectively, $\add(\nu_B\cpx{\bar{T}})$).
   \label{complexTX}
\end{Lemma}
\begin{proof}
   By the proof of Theorem \ref{Theorem 1}, we see that $F\simeq
   F_2F_1^{-1}$ for two almost $\nu$-stable derived equivalences
   $F_1:\Db{C}\lra\Db{A}$ and $F_2:\Db{C}\lra\Db{B}$. For each
   $A$-module $X$, by \cite[Lemma 3.2]{HuXi3} and the definition of
   almost $\nu$-stable derived equivalences, we see that
   $F_1^{-1}(X)$ is isomorphic in $\Db{C}$ to a complex $\cpx{Q_X}$
   with $Q_X^i=0$ for all $i>0$ and $Q_X^i\in\nust{C}$ for all
   $i<0$. Applying $F_2$ to the distinguished triangle
   $(\sigma_{<0}\cpx{Q_X})[-1]\lra Q_X^0\lra \cpx{Q_X}\lra
   \sigma_{<0}\cpx{Q_X}$,
   we get a distinguished triangle in $\Db{B}$
   $$F_2(\sigma_{<0}\cpx{Q_X})[-1]\lra F_2(Q_X^0)\lra F_2(\cpx{Q_X})\lra F_2(\sigma_{<0}\cpx{Q_X}).$$
   Since $(\sigma_{<0}\cpx{Q_X})[-1]$ is a complex in
   $\Kb{\nust{C}}$, by Lemma \ref{lemmaEtoE}, the complex
   $F_2(\sigma_{<0}\cpx{Q_X})[-1]$ is isomorphic in $\Db{B}$ to a
   complex $\cpx{U}$ in $\Kb{\nust{B}}$. By \cite[Lemma 3.1]{HuXi3}
   and the definition of almost $\nu$-stable derived equivalences,
   the complex $F_2(Q_X^0)$ is isomorphic in $\Db{B}$ to a complex
   $\cpx{V}$ with $V^i\in\nust{B}$ for all $i>0$ and $V^i=0$ for all
   $i<0$. Thus, the complex $F_2(\cpx{Q_X})$, which is isomorphic in $\Db{B}$ to
   $F(X)$, is isomorphic in $\Db{B}$ to the mapping cone $\con(\alpha)$ of a chain map
   $\alpha$ from $\cpx{U}$ to $\cpx{V}$. Now it is clear that all
   the terms of $\con(\alpha)$ in non-zero degrees are in
   $\nust{B}$. Taking a radical complex $\cpx{\bar{T}_X}$ which is
   isomorphic to $\con(\alpha)$ in $\Kb{B}$, we see that $F(X)$ is
   isomorphic to $\cpx{\bar{T}_X}$ and $\bar{T}_X^{\pm}\in\nust{B}$.

   Suppose that $\cpx{W}$ is another radical complex with
   $W^{\pm}\in\nust{B}$, and $F(X)\simeq
   \cpx{W}$. Then $\cpx{W}$ and $\cpx{\bar{T}_X}$ are isomorphic in
   $\Db{B}$. By Lemma \ref{LemmaKDiso}, they are isomorphic in
   $\Kb{B}$. Since both $\cpx{W}$ and $\cpx{\bar{T}_X}$ are radical
   complexes, they are also isomorphic in $\Cb{B}$.

   Since all the complexes
   in $\add(\cpx{\bar{T}})$ and $\add(\nu_B\cpx{\bar{T}})$ have the desired
   form, the last statement follows by the uniqueness of
   $\cpx{\bar{T}_X}$.
\end{proof}

In the following, without loss of generality, we fix for each
$A$-module $X$ a complex $\cpx{\bar{T}_X}$ defined in Lemma
\ref{complexTX} and assume that  $F(X)=\cpx{\bar{T}_X}$ for all
$A$-modules $X$. Let $X$ and $Y$ be two $A$-modules. There is a
natural isomorphism
$$\Hom_A(X, Y)\simeq \Hom_{\Db{B}}(\cpx{\bar{T}_X}, \cpx{\bar{T}_Y})$$
sending $f$ to $F(f)$. By Lemma \ref{LemmaKDiso}, there is a natural
isomorphism $$\Hom_{\Kb{B}}(\cpx{\bar{T}_X}, \cpx{\bar{T}_Y})\simeq
\Hom_{\Db{B}}(\cpx{\bar{T}_X}, \cpx{\bar{T}_Y})$$ induced by the
localization functor from $\Kb{B}$ to $\Db{B}$. It is easy to see
that there is a natural map
$$\Hom_{\Kb{B}}(\cpx{\bar{T}_X}, \cpx{\bar{T}_Y})\lra \nustHom_B(\bar{T}_X^0, \bar{T}_Y^0)$$
sending $\cpx{u}$ to $\underline{u}^0$. Indeed, if $\cpx{u}=\cpx{v}$
in $\Hom_{\Kb{B}}(\cpx{\bar{T}_X}, \cpx{\bar{T}_Y})$, then $u^0-v^0$
factorizes through $\bar{T}_X^1\oplus\bar{T}_Y^{-1}$ which is in
$\nust{B}$ by definition. This means
$\underline{u}^0-\underline{v}^0=0$ in $\nustHom_B(\bar{T}_X^0,
\bar{T}_Y^0)$. Altogether, we have a natural morphism
$$\phi: \Hom_A(X, Y)\lra \nustHom_B(\bar{T}_X^0, \bar{T}_Y^0)$$
sending $f$ to $\underline{u}^0$, where $\cpx{u}$ is a chain map
such that $\cpx{u}=F(f)$. Now if $f$ factorizes through an
$A$-module in $\nust{A}$, then $\cpx{u}$ factorizes through a
complex $\cpx{P}$ in $\Kb{\nust{B}}$ by Lemma \ref{lemmaEtoE}. By
Lemma \ref{LemmaKDiso}, we can assume that $\cpx{u}=\cpx{g}\cpx{h}$
in $\Kb{B}$ for chain maps $\cpx{g}: \cpx{\bar{T}_X}\lra \cpx{P}$
and $\cpx{h}: \cpx{P}\lra \cpx{\bar{T}_Y}$. Thus, it follows that
$u^0-g^0h^0$ factorizes through $\bar{T}_X^1\oplus\bar{T}_Y^{-1}$,
and consequently $u^0$ factorizes through $P^0\oplus
\bar{T}_X^1\oplus\bar{T}_Y^{-1}$ which is in $\nust{B}$. Hence
$\underline{u}^0=0$. Hence we get a natural morphism
$$\bar{\phi}: \nustHom_A(X, Y)\lra \nustHom_B(\bar{T}_X^0, \bar{T}_Y^0)$$
Now we define a functor $\phi_F: \nustcat{A}\lra \nustcat{B}$. For
each $A$-module $X$, we set $\phi_F(X):=\bar{T}_X^0$, and for each
morphism $\underline{f}\in\nustHom_A(X, Y)$, we define
$\phi_F(\underline{f}):=\underline{u}^0$, where $\cpx{u}=F(f)$. Now
it is easy to see that the diagram
$$\xymatrix{
  \nustcat{A}\ar[r]^(.35){\Sigma}\ar[d]^{\phi_F} & \Db{A}/\Kb{\nust{A}}\ar[d]^{F}\\
\nustcat{B}\ar[r]^(.35){\Sigma} & \Db{B}/\Kb{\nust{B}}
 }\quad \quad (\clubsuit)$$
is commutative up to isomorphism. Indeed, one can check that the
isomorphism $\cpx{\bar{T}_X}\llaf{s}\sigma_{\geq
0}\cpx{\bar{T}_X}\rlaf{t}X$ in $\Db{B}/\Kb{\nust{B}}$ with $s$ and
$t$ the canonical maps is a natural map, and this gives rise to an
 isomorphism from the functor $F\Sigma$ to the functor
$\Sigma\phi_F$.

\medskip
 For an Artin algebra, in the
following theorem, we denote by $\stimodcat{A}$ the quotient
category of $\modcat{A}$ modulo morphisms factorizing through
injective modules.
\begin{Theorem}
   Let $F:\Db{A}\lra\Db{B}$ be an iterated almost $\nu$-stable
   derived equivalence. Then we have the following:

   $(1)$ The functor $\phi_F: \nustcat{A}\lra\nustcat{B}$ is an
   equivalence;

   $(2)$ The functor $\phi_F$ induces an equivalence between
   $\stimodcat{A}$ and $\stimodcat{B}$;

   $(3)$ The functor $\phi_F$ induces an equivalence between $\stmodcat{A}$ and
   $\stmodcat{B}$;

   $(4)$ The functor $\phi_F$ is uniquely (up to
   isomorphism) determined by the commutative diagram $(\clubsuit)$.
   Moreover, if $F': \Db{B}\lra\Db{C}$ is another iterated
   almost $\nu$-stable derived equivalence, then $\phi_{F'F}\simeq
   \phi_{F'}\phi_F$.\label{TheoremStableFunctor}
\end{Theorem}

\begin{proof}
  Let $G$ be a quasi-inverse of $F$. Then $G$ also induces an
  equivalence between $\Db{B}/\Kb{\nust{B}}$ and
  $\Db{A}/\Kb{\nust{A}}$. We also denote it by $G$. Then by the
  above commutative diagram of additive functors, the functor
  $\Sigma\phi_G\phi_F$ is isomorphic to the functor $GF\Sigma$,
  which is isomorphic to $\Sigma$. By Proposition
  \ref{PropositionEmbedding}, the functor $\Sigma$ is a fully faithful
  embedding. Hence $\phi_G\phi_F$ is isomorphic to
  $1_{\nustcat{A}}$.  By symmetry, the functor $\phi_F\phi_G$ is
  also isomorphic to $1_{\nustcat{B}}$. Hence $\phi_F$ is
  an equivalence, and $(1)$ is proved.

 By the construction of $\phi_F$, it follows from Lemma
 \ref{complexTX} that $\phi_F$ sends projective modules to
 projective modules, and sends injective modules to injective
 modules. Moreover, the modules in $\nust{A}$ and $\nust{B}$ are all
 projective-injective. Thus, the statements $(2)$ and $(3)$ follow.

 (4) If $\phi: \nustcat{A}\lra\nustcat{B}$ is a functor such that $\Sigma\phi\simeq
 F\Sigma$, then the functor $\Sigma\phi$ is isomorphic to
 $\Sigma\phi_F$. Hence $\phi\simeq \phi_F$ since $\Sigma$ is fully
 faithful. The rest of (4) follows similarly.
\end{proof}

\medskip
{\it Remark: } (1) It follows from the definition of iterated almost
$\nu$-stable derived equivalences and \cite[Theorem 3.7]{HuXi3} that
every iterated almost $\nu$-stable derived equivalence induces an
equivalence between the stable module categories, however, the proof
of Theorem \ref{TheoremStableFunctor} presented here is not based on
the earlier result \cite[Thoerem 3.7]{HuXi3}, and is completely
different from the proof there. Moreover, Theorem
\ref{TheoremStableFunctor} is more general than \cite[Theorem
3.7]{HuXi3} since we get an equivalence between $\nustcat{A}$ and
$\nustcat{B}$ which is not obtained in \cite[Theorem 3.7]{HuXi3}.

(2) In case that $F$ is an almost $\nu$-stable derived equivalence,
it follows by definition that the stable equivalence from
$\stmodcat{A}$ to $\stmodcat{B}$ induced by the functor $\phi_F$
coincides with the stable functor $\bar{F}$ considered in
\cite{LiuXi3}.

\section{Constructions of iterated almost $\nu$-stable derived equivalences}

In this section, we shall give some constructions of iterated almost
$\nu$-stable derived equivalences.

Let us recall from \cite{AS} the definition of approximations. Let
$\cal C$ be a category, and let $\cal D$ be a full subcategory of
$\cal C$, and $X$ an object in $\cal C$. A morphism $f: D\lra X$ in
$\cal C$ is called a \emph{right} $\cal D$-\emph{approximation} of
$X$ if $D\in {\cal D}$ and the induced map Hom$_{\cal C}(D',f)$:
Hom$_{\cal C}(D',D)\lra$ Hom$_{\cal C}(D',X)$ is surjective for
every object $D'\in {\cal D}$. Dually, one can define \emph{left}
$\cal D$-\emph{approximations}.

 By Theorem \ref{Theorem 1}, to get an iterated
almost $\nu$-stable derived equivalence, we only need to construct a
derived equivalence with the involved tilting complexes satisfying
the conditions in Theorem \ref{Theorem 1}. Let $A$ be an algebra,
and let $P, Q$ be two projective $A$-modules satisfying the
following two conditions:

\smallskip
    (1) $\add(_AP)=\add(\nu_AP)$, $\add(_AQ)=\add(\nu_AQ)$;

    (2) $\Hom_A(P, Q)=0.$

\smallskip
{\parindent=0pt For each positive integer $r$, we can form the
following complex}:
$$0\lra P^{-r}\lraf{f_r}\lra P^{-r+1}\lra\cdots\lra P^{-1}\lraf{f_1} A\lra 0, $$
where $f_1: P^{-1}\lra A$ is a right $\add(_AP)$-approximation of
$A$, and $f_{i+1}: P^{-i-1}\lra P^{-i}$ is a right
$\add(_AP)$-approximation of $\Ker(f_i)$ for $i=1,\cdots, r-1$.
Similarly, we can form a complex
$$0\lra A\lraf{g_1} Q^1\lra \cdots\lra Q^{s-1}\lraf{g_s} Q^s\lra 0, $$
where $g_1$ is a left $\add(_AQ)$-approximation of $A$, and
$g_{i+1}$ is a left $\add(_AQ)$-approximation of $\Coker (g_i)$ for
$i=1,2,\cdots, s-1$. Since $\Hom_A(P, Q)=0$, connecting  the two
complexes together, we get a complex
$$0\lra P^{-r}\lra\cdots\lra P^{-1}\lraf{f_1} A\lraf{g_1} Q^1\lra\cdots\lra Q^s\lra 0, $$
where $A$ is in degree zero. We denote this complex by
$\cpx{T_{P,Q}}$, and let $\cpx{T}:=\cpx{T_{P, Q}}\oplus P[r]\oplus
  Q[-s]$.
\begin{Proposition}\label{ApproxConstruction}
  Keeping the notations above, we have the following:

  $(1)$ The complex $\cpx{T}$ is a tilting complex.

  $(2)$ Let $B:=\End_{\Db{A}}(\cpx{T})$. Then $\cpx{T}$ induces an
  iterated almost $\nu$-stable derived equivalence between the algebras  $A$ and
  $B$.
\end{Proposition}
\begin{proof}
(1) By the construction of $\cpx{T}$, we have
$$T^i=\left\{
  \begin{array}{ll}
    P^{-r}\oplus P, & {i=-r}\,; \\
    P^{i}, & {-r<i<0}\,; \\
    A, & {i=0\,;} \\
    Q^i, & {0<i<s\,;} \\
    Q^s\oplus Q, & {i=s\,;}\\
   0 & \hbox{otherwise.}

  \end{array}
\right., \mbox{ and } d_T^i= \left\{
  \begin{array}{ll}
    \genfrac[]{0pt}{1}{f_r}{0}, & {i=-r}\,; \\
    f_{-i}, & {-r<i<0}\,; \\
    g_{i+1}, & {0\leq i<s-1\,;} \\
    \genfrac[]{0pt}{1}{g_s}{0}, & {i=s-1\,;}\\
   0 & \hbox{otherwise.}
  \end{array}
\right.$$

 We first show that $\Hom_{\Kb{\pmodcat{A}}}(\cpx{T},\cpx{T}[i])=0$
 for all $i\neq 0$. Assume that $i$ is a positive integer. Let
 $\cpx{u}$ be a morphism in
 $\Hom_{\Kb{\pmodcat{A}}}(\cpx{T},\cpx{T}[i])$. Then we have the
 following commutative diagram
$$\xymatrix{
  \cdots \ar[r] & T^{-i-1}\ar[r]^{d_T^{-i-1}}\ar[d]^{u^{-i-1}} & T^{-i}\ar[r]^{d_T^{-i}}\ar[d]^{u^{-i}} &
  T^{-i+1}\ar[r]^{d_T^{-i+1}}\ar[d]^{u^{-i+1}}
  &\cdots\ar[r] & T^{-1}\ar[r]^{d^{-1}}\ar[d]^{u^{-1}} & T^0\ar[r]^{d_T^0}\ar[d]^{u^0} & T^1\ar[r]^{d_T^1}\ar[d]^{u^1} &\cdots\\
    \cdots \ar[r] & T^{-1}\ar[r]^{d_T^{-1}} & T^{0}\ar[r]^{d_T^0} &
    T^{1}\ar[r]^{d_T^1}
  &\cdots\ar[r] & T^{i-1}\ar[r]^{d_T^{i-1}} & T^i\ar[r]^{d_T^i} & T^{i+1}\ar[r]^{d_T^{i+1}} &\cdots
}$$ Since $\Hom_A(P, Q)=0$, we have $u^k=0$ for all $-i<k<0$. By
definition, $T^{-i}\in\add(_AP)$. Since $d_T^{-1}=f_1$ is a right
$\add(_AP)$-approximation, there is a map $h^{-i}: T^{-i}\lra
T^{-1}$ such that $u^{-i}=h^{-i}d_T^{-1}$. Thus,
$(u^{-i-1}-d_T^{-i-1}h^{-i})d_T^{-1}=d_T^{-i-1}u^{-i}-d_T^{-i-1}h^{-i}d_T^{-1}=d_T^{-i-1}u^{-i}-d_T^{-i-1}u^{-i}=0$.
Since $d_T^{-2}$ is a right $add(_AP)$-approximation of $\Ker
(d_T^{-1})$, there is a map $h^{-i-1}: T^{-i-1}\lra T^{-2}$ such
that $u^{-i-1}-d_T^{-i-1}h^{-i}=h^{-i-1}d_T^{-2}$, that is
$u^{-i-1}=d_T^{-i-1}h^{-i}+h^{-i-1}d_T^{-2}$. Similarly, for each
integer $k<-i-1$, there are maps $h^{k+1}: T^{k+1}\lra T^{k+i}$ and
$h^{k}: T^{k}\lra T^{k+i-1}$ such that
$u^k=d_T^kh^{k+1}+h^kd_T^{k+i-1}$. Defining $h^k=0$ for all
$-i<k\leq 0$, we have $u^k=d_T^kh^{k+1}+h^kd_T^{k+i-1}$ for all
$k<0$. Similarly, we can prove that
$u^k=d_T^kh^{k+1}+h^kd_T^{k+i-1}$ for $k\geq 0$. Altogether, we have
shown that $\cpx{u}=0$ in $\Kb{\pmodcat{A}}$. Hence
$\Hom_{\Kb{\pmodcat{A}}}(\cpx{T},\cpx{T}[i])=0$ for all $i>0$. By an
analogous proof, we have
$\Hom_{\Kb{\pmodcat{A}}}(\cpx{T},\cpx{T}[i])=0$ for all $i<0$.
Finally, since $P[r]$ and $Q[-s]$ are in $\add(\cpx{T})$, we deduce
that $_AA$ is in the triangulated subcategory of $\Kb{\pmodcat{A}}$
generated by $\add(\cpx{T})$. Hence $\add(\cpx{T})$ generates
$\Kb{\pmodcat{A}}$ as a triangulated category, and consequently
$\cpx{T}$ is a tilting complex over $A$.

(2) Let $F:\Db{{A}}\lra\Db{{B}}$ be a derived equivalence induced by
$\cpx{T}$. Also, we use $F$ to denote the equivalence between
$\Kb{\pmodcat{A}}$ and $\Kb{\pmodcat{B}}$ induced by $F$. Set
$\bar{P}:=\Hom_{\Kb{\pmodcat{A}}}(\cpx{T}, P[r])\simeq F(P[r])$,
$\bar{Q}:=\Hom_{\Kb{\pmodcat{A}}}(\cpx{T}, Q[-s])\simeq F(Q[-s])$,
and $\bar{U}:=\Hom_{\Kb{\pmodcat{A}}}(\cpx{T},\cpx{T_{P,Q}})\simeq
F(\cpx{T_{P,Q}})$. Then $F(P)\simeq \bar{P}[-r]$ and $F(Q)\simeq
\bar{Q}[s]$.  For simplicity, we list some subcomplexes of
$\cpx{T}$:
$$\begin{array}{ll}
\cpx{P}: & 0\lra P^{-r}\lra\cdots\lra P^{-1}\lra 0,\\
\cpx{Q}: & 0\lra Q^{1}\lra\cdots\lra Q^{s}\lra 0,\\
\cpx{R}: & 0\lra A\lra Q^{1}\lra\cdots\lra Q^{s}\lra 0.
\end{array}$$
By Lemma \ref{lemmaFormofFX}, $F(\cpx{P})$ is isomorphic to a
complex $\cpx{\bar{P}}$ in $\Kb{\add(\bar{P})}$ such that
$\bar{P}^i=0$ for all $i<0$ and all $i\geq r$. The complex
$F(\cpx{Q})$ is isomorphic to a complex $\cpx{\bar{Q}}$ in
$\Kb{\add(\bar{Q})}$ with $\bar{Q}^i=0$ for all $i>0$ and all $i\leq
-s$. Note that there is a distinguished triangle in
$\Kb{\pmodcat{A}}$
$$\cpx{P}[-1]\lra\cpx{R}\lra \cpx{T_{P,Q}}\lra \cpx{P}.$$
Applying $F$, we get a distinguished triangle in $\Kb{\pmodcat{B}}$:
$$F(\cpx{P})[-1]\lra F(\cpx{R})\lra  F(\cpx{T_{P,Q}})\lra F(\cpx{P}).$$
Hence $F(\cpx{R})$ is isomorphic to a complex of the following form:
$$0\lra \bar{U}\lra \bar{P}^0\lra\bar{P}^{1}\lra\cdots$$
with $\bar{U}$ in degree $0$. Next we have a distinguished triangle
$$\cpx{Q}\lra\cpx{R}\lra A\lra \cpx{Q}[1]$$
in $\Kb{\pmodcat{A}}$. Applying $F$, we see that $F(A)$ is
isomorphic to a complex $\cpx{\bar{T}}$ of the form
$$\cdots\lra \bar{Q}^{-1}\lra \bar{Q}^0\lra \bar{U}\lra \bar{P}^0\lra\bar{P}^{1}\lra\cdots, $$
where $\bar{U}$ is in degree zero. Note that $\cpx{\bar{T}}$ is a
tilting complex associated to $F^{-1}$ since
$F^{-1}(\cpx{\bar{T}})\simeq A$. Since $\add(_AP)=\add(\nu_AP)$ and
$\add(_AQ)=\add(\nu_AQ)$, we have
$\add(_B\bar{P})=\add(\nu_B\bar{P})$ and
$\add(_B\bar{Q})=\add(\nu_B\bar{Q})$. Thus, we have
$\add(_AT^{\pm})=\add(\nu_AT^{\pm})$ and
$\add(_B\bar{T}^{\pm})=\add(\nu_B\bar{T}^{\pm})$. By Theorem
\ref{Theorem 1}, the statement $(2)$ follows.
\end{proof}

To illustrate Proposition \ref{ApproxConstruction}, we give an
example. Let $A$ be the finite-dimensional $k$-algebra given by the
quiver
$$\xymatrix{
\bullet\ar@<2pt>[r]^{\alpha}
&\bullet\ar@<2pt>[l]^(1){1}^(0){2}^{\alpha'}\ar@<2pt>[r]^{\beta}
&\bullet\ar@<2pt>[l]^(0){3}^{\beta'}\ar@<2pt>[r]^{\gamma}
&\bullet\ar@<2pt>[l]^(0){4}^{\gamma'} }$$ with relations
$\alpha'\alpha=\beta\beta'=\alpha\beta=\beta\gamma=\beta'\alpha'=\gamma'\beta'=
\beta'\beta-\gamma\gamma'=0$. We use $P_i$ to denote the
indecomposable projective $A$-module  corresponding to the vertex
$i$ for $i=1,2,3,4$. The Loewy structure of the projective
$A$-modules can be listed as follows.
$$
P_1: \begin{array}{c}1\\2\\1\end{array}\quad P_2:
\begin{array}{ccc}& 2 \\ 1 & & 3\\&&\end{array}\quad P_3: \begin{array}{ccc}& 3 \\ 2 & &
4\\&3&\end{array}\quad P_4: \begin{array}{c}4\\3\\4\end{array}
$$
Let $P:=P_1$ and $Q:=P_3\oplus P_4$. Then we have
$\add(_AP)=\add(\nu_AP)$, $\add(_AQ)=\add(\nu_AQ)$, and $\Hom_A(P,
Q)=0$. Using Proposition \ref{ApproxConstruction}, we have a tilting
complex $\cpx{T}$ over $A$. The indecomposable direct summands of
$\cpx{T}$ are:
$$\begin{array}{rl}
\cpx{T_1}: & 0\lra P_1\lra 0\\
\cpx{T_2}: & 0\lra P_1\lra P_2\lra P_3\lra 0\\
\cpx{T_3}: & \hspace{2.3cm} 0\lra P_3\lra 0\\
\cpx{T_4}: & \hspace{2.3cm} 0\lra P_4\lra 0\\
\end{array}$$
A calculation shows that the algebra $B:=\End_{\Db{A}}(\cpx{T})$ is
 given by the quiver
$$
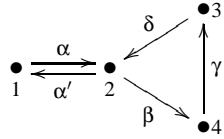
\xymatrix@R=4mm{
&&\bullet \ar[ld]_{\delta} \\
\bullet\ar@<2pt>[r]^{\alpha} &
\bullet\ar@<2pt>[l]^(1){1}^(0){2}^{\alpha'}\ar[rd]_{\beta} \\
&& \bullet \ar[uu]_(0){4}_(1){3}_{\gamma} }$$ with relations
$\alpha'\alpha=\alpha\beta=\delta\alpha'=\beta\gamma\delta=\gamma\delta\beta\gamma=0$.
By Proposition \ref{ApproxConstruction}, $\cpx{T}$ induces an
iterated almost $\nu$-stable derived equivalence between $A$ and
$B$. Therefore, $A$ and $B$ are also stably equivalent of Morita
type.

\medskip
The following proposition shows how we can construct iterated almost
$\nu$-stable derived equivalences inductively.
\begin{Proposition}
 Let $F:\Db{A}\lra\Db{B}$ be an iterated almost $\nu$-stable derived
 equivalence between two finite-dimensional algebras $A$ and $B$ over a field $k$, and let
 $\phi_{F}$ be the stable equivalence induced by $F$ (see, {\rm Theorem \ref{TheoremStableFunctor}}). Then we have
 the following:

 $(1)$ For each $A$-module $X$, there is an iterated almost
 $\nu$-stable derived equivalence between the endomorphism algebras $\End_A(A\oplus X)$ and $\End_B(B\oplus
 \phi_F(X))$;

 $(2)$ For a finite-dimension self-injective $k$-algebra $C$, there
 is an iterated almost $\nu$-stable derived equivalence between
 $A\otimes_kC$ and $B\otimes_kC$.\label{corollaryInductive}
\label{InductiveConstruction}
\end{Proposition}
\begin{proof}
Suppose that $F\simeq F_1F_2\cdots F_n$ such that $F_i$ or
$F_i^{-1}$ is almost $\nu$-stable for all $i$. By Theorem
\ref{TheoremStableFunctor}, we have $\phi_F\simeq
\phi_{F_1}\phi_{F_2}\cdots\phi_{F_n}$. By the remark after Theorem
\ref{TheoremStableFunctor}, we know that $\phi_{F_i}$ coincides with
the $\bar{F}_i$ considered in \cite{HuXi3} for all $i$. Thus, the
statements (1) follows  from \cite[Corollary 1.3]{HuXi3}. The proof
of (2) is similar to that of \cite[Proposition 6.2]{HuXi3}.
\end{proof}

Let us recall from \cite{HuXi4} the definition of
$\Phi$-Auslander-Yoneda algebras. A subset $\Phi$ of the set of
natural numbers $\mathbb{N}$ is called admissible provided that: (1)
$0\in \Phi$; (2) If $i+j+k\in\Phi$ for $i, j, k\in\Phi$, then
$i+j\in\Phi$ implies that $j+k\in\Phi$. For instance, the sets
$\mathbb{N}$, $\{0, 1, \cdots, n\}$ are admissible subsets of
$\mathbb{N}$. Suppose that $\Phi$ be an admissible subset of
${\mathbb N}$. Let $A$ be an Artin algebra, and let $X$ be an
$A$-module.  Now we consider the Yoneda algebra $\Ext_A^*(X,
X)=\bigoplus_{i\geq 0}\Hom_{\Db{A}}(X, X[i])$ of $X$, and define
$\Ex{A}{\Phi}{X}:=\bigoplus_{i\in\Phi}\Hom_{\Db{A}}(X, X[i])$ with
multiplication: for $a_i\in\Hom_{\Db{A}}(X, X[i])$ and
$a_j\in\Hom_{\Db{A}}(X, X[j])$, we define $a_i\cdot a_j=a_ia_j$ if
$i+j\in\Phi$, and zero otherwise. Then one can check that
$\Ex{A}{\Phi}{X}$ is an associated algebra. If $\Phi=\{0\}$, then
$\Ex{A}{\Phi}{X}$ is isomorphic to $\End_A(X)$. If
$\Phi=\mathbb{N}$, then $\Ex{A}{\Phi}{X}$ is just the Yoneda algebra
of $X$.
\begin{Proposition}
   Let $F:\Db{A}\lra\Db{B}$ be an iterated almost $\nu$-stable
   derived equivalence between two Artin algebras $A$ and $B$.
   Suppose that $\Phi$ is an admissible subset of $\mathbb{N}$. Then
   we have the following:

   $(1)$ For any $A$-module $X$, there is a derived equivalence
   between the $\Phi$-Auslander-Yoneda algebras $\Ex{A}{\Phi}{A\oplus X}$ and $\Ex{B}{\Phi}{B\oplus
   \phi_F(X)}$;

   $(2)$ If $\Phi$ is a finite set, then for any $A$-module $X$,
   there is an iterated almost $\nu$-stable derived equivalence
   between $\Ex{A}{\Phi}{A\oplus X}$ and $\Ex{B}{\Phi}{B\oplus
   \phi_F(X)}$.
\end{Proposition}
\begin{proof}
 Using the result \cite[Theorem 3.4]{HuXi4}, the proof is similar to
 that of Proposition \ref{InductiveConstruction} (1).
\end{proof}

\bigskip {\parindent=0pt\Large\bf Acknowledgement}

\medskip
This work is partially supported by China Postdoctoral Science
Foundation (No. 20080440003). The author also thanks the Alexander
von Humboldt Foundation for support.

\medskip
{\small {\sc Wei Hu},

School of Mathematical Sciences, Beijing Normal University, Beijing,
100875, China
\end{document}